\documentclass[10pt,a4paper,oneside]{amsproc}

\usepackage{latexsym, amsthm}

\usepackage{amsfonts, amsmath, amssymb}

\usepackage{euscript,mathrsfs}

\newtheorem{theorem}{Theorem}[section]
\newtheorem{lemma}[theorem]{Lemma}
\newtheorem{proposition}[theorem]{Proposition}
\newtheorem{conjecture}[theorem]{Conjecture}
\newtheorem{corollary}[theorem]{Corollary}

\theoremstyle{definition}
\newtheorem{definition}[theorem]{Definition}

\theoremstyle{remark}
\newtheorem{remark}[theorem]{Remark}
\def\Fq{{\mathbb F}_q}

\def\AA{{\mathbb A}}
\def\FF{{\mathbb F}}
\def\PP{{\mathbb P}}
\def\hpm{\widehat{\mathbb P}^m}
\def\p{\mathfrak{p}}

\newcommand{\codim}{\operatorname{codim}}

\newcommand{\PRM}{\mathrm{PRM}}

\newcommand{\V}{\mathsf{V}}
\newcommand{\Proj}{\mathrm{Proj}}
\newcommand{\Vmd}{\mathscr{V}_{m,d}}

\def\PP{{\mathbb P}}

\def\Z{{\mathbb Z}}

\begin{document}

\title[Systems of Homogeneous Polynomials over Finite Fields]{Number of Solutions of Systems of Homogeneous Polynomial Equations over Finite Fields}
\author{Mrinmoy Datta}
\address{Department of Mathematics,
Indian Institute of Technology Bombay,\newline \indent
Powai, Mumbai 400076, India.}
\curraddr{Department of Applied Mathematics and Computer Science, \newline \indent
Technical University of Denmark, DK 2800, Kgs. Lyngby, Denmark}
\email{mrinmoy.dat@gmail.com}
\thanks{The first named author was partially supported by a doctoral fellowship from the National Board for Higher Mathematics, a division of the Department of Atomic Energy, Govt. of India.}

\author{Sudhir R. Ghorpade}
\address{Department of Mathematics, 
Indian Institute of Technology Bombay,\newline \indent
Powai, Mumbai 400076, India.}
\email{srg@math.iitb.ac.in}
\thanks{The second named author was partially supported by Indo-Russian project INT/RFBR/P-114 from the Department of Science \& Technology, Govt. of India and  IRCC Award grant 12IRAWD009 from IIT Bombay.}

\subjclass[2010]{Primary 14G15, 11T06, 11G25, 14G05 Secondary 51E20, 05B25}

\date{}

\begin{abstract}
We consider the problem of determining  the maximum number of common zeros in a projective space over a finite field for a system of linearly independent multivariate homogeneous polynomials defined over that field. There is an elaborate conjecture of Tsfasman and Boguslavsky that predicts the maximum value when the homogeneous polynomials have the same degree that is not too large in comparison to the size of the finite field. We show that this conjecture holds in the affirmative if the number of polynomials does not exceed the total number of variables.  This extends the results of Serre (1991) and Boguslavsky (1997) for the case of one and two polynomials, respectively. Moreover, it complements our recent result that the conjecture is false, in general, if the number of polynomials exceeds the total number of variables.  
\end{abstract}

\maketitle


\section{Introduction}
\label{sec:in}

Let $r,d,m$ be positive integers and let $\Fq$ denote the finite field with $q$ elements. Also let $S:=\Fq[x_0,x_1, \dots , x_m]$ denote the ring of polynomials in $m+1$ variables with coefficients in $\Fq$ and $\PP^m = \PP^m(\Fq)$ the $m$-dimensional projective space over $\Fq$. We are interested 
in the following question. 

\medskip

{\bf Question:} What is the maximum number of common zeros that a system of $r$ linearly independent homogeneous polynomials of degree $d$ in $S$ 
can have in $\PP^m(\Fq)$?

\medskip

Note that because of the condition  of  linear independence, the question is meaningful when  $r\le M$, where  $M:= \binom{m+d}{d}$. Also note that if $\Vmd$ denotes the Veronese variety given by the image of $\PP^m$ in $\PP^{M-1}$ under the Veronese map of degree $d$, then the question is equivalent to the following:

\medskip

{\bf Question:} What is the maximum number of $\Fq$-rational points that a section of $\Vmd$ by a linear subspace of $\PP^{M-1}$ of codimension $r$ can have?

\medskip
In case $d\ge q+1$,  
it is easy to construct for many values of $r$, systems of  $r$ linearly independent homogeneous polynomials of degree $d$ in $S$ which vanish at every point of $\PP^m(\Fq)$.  (See Remark \ref{qplusone} for details.) 
So for most values of $r$ (and certainly for $r\le m+1$), the answer in the case $d\ge q+1$ is 
$p_m$, where for any $k\in \Z$, we set $p_k := |\PP^k(\Fq)| = q^k + q^{k-1} + \dots + q + 1$ if $k \ge 0$ and $p_k:=0$ if $k < 0$. Thus the question is mainly of interest when $d\le q$, and we will mainly restrict to this case.


A brief history of the above question is as follows. It was first posed by Tsfasman in the late 1980's in the case $r=1$, i.e., for hypersurfaces in $\PP^m$; in fact, Tsfasman conjectured that the maximum value is $dq^{m-1} + p_{m-2}$ when $r=1$ and $d\le q+1$.  This conjecture was proved in the affirmative by Serre \cite{Se} and independently, by S{\o}rensen  \cite{So}  in 1991 (see also \cite{DG}). The next advance came in 1997 when  Boguslavsky \cite{Bog} gave a complete answer in the case $r=2$ and $d< q-1$. Yet another decisive step was taken, albeit in disguise, by Zanella \cite{Z} who solved in 1998 the equivalent question for sections of the Veronese variety given by the quadratic Veronese embedding of $\PP^m$, i.e., in the case $d=2$. 
In \cite{Bog}, Boguslavsky 
also gave a number of conjectures related to the general question, ascribing
some of them to Tsfasman. Surmising from these conjectures and accompanying results, 
one has 
a plausible answer to the above question, at least when $d < q-1$. 

\medskip

{\bf Tsfasman-Boguslavsky Conjecture (TBC)}: Assume that $r \le \binom{m+d}{d}$ and $d < q-1$. Then 
the maximum number of common zeros that a system of $r$ linearly independent homogeneous polynomials of degree $d$ in $S$ 
can have in $\PP^m(\Fq)$ is
\begin{equation}
\label{TBr}
T_r(d,m): = \displaystyle{   p_{m-2j} + \sum_{i=j}^m} \nu_i (p_{m-i} - p_{m-i-j}),
\end{equation}
where $(\nu_1, \dots, \nu_{m+1})$ is the $r$th element in descending lexicographic order among 
$(m+1)$-tuples $(\alpha_1, \dots , \alpha_{m+1})$ of nonnegative integers satisfying 
$\alpha_1+ \cdots + \alpha_{m+1} = d$, and where $j := \min\{i : \nu_i \ne 0\}$. 

The results of Serre \cite{Se} and Boguslavsky \cite{Bog} prove the TBC in the affirmative when $r\le 2$. 
But for $r > 2$ the question remained open for a considerable time. The aim of this paper is to prove that 
the TBC holds in the affirmative for any $r \le m+1$. (See Theorem \ref{mainthm} for a precise statement.)  Our proof uses the result of Serre \cite{Se}, 
but not of Boguslavsky \cite{Bog}. Thus Boguslavsky's theorem becomes a corollary. It should be remarked that an affirmative answer to the TBC in the case 
$r \le m+1$  is perhaps the best one can expect since we 
have shown in \cite{DG} that the TBC is false, in general, if $r > m+1$. However, the question posed at the beginning of the paper is still valid for $r > m+1$, and we propose in Section \ref{MaxFam} a new conjecture for many (but not all) values of $r$ beyond $m+1$. This is partly motivated by an affine analogue of this question and the definitive work on it by Heijnen and Pellikaan \cite{HP}. We also remark that our results on the TBC give bounds on the number of $\Fq$-rational points of projective algebraic varieties in $\PP^m$ defined by $m+1$ or fewer equations of the same degree, and these bounds are easy to use in practice (one just needs to check that the equations are linearly independent) and are also optimal because they are sometimes attained. However, if one has additional (and not-so-easily-checkable) information on the variety such as the dimensions and degrees of  its irreducible components, then there are alternate bounds given recently by Couvreur \cite{C}, and these bounds are sometimes better. 
We refer to \cite[\S 4.2]{DG} for a comparison of our bounds with those of Couvreur. Moreover, if the variety is known to be irreducible (and better still, nonsingular), then there are other general bounds such as those of Lang and Weil, and also those that arise from Weil conjectures. We refer to \cite{GL2} and the references therein for more on these general bounds.

This paper is organized as follows. The next section introduces basic notation and contains a discussion of the initial cases (when $d$, $m$, or $r$ equals $1$) as well as an affine variant of the question posed above, and some useful facts about projective varieties and complete intersections over finite fields. An elementary, but useful, notion of a coprime close family of homogeneous polynomials is introduced in Section \ref{comb}, and a consequence of a combinatorial structure theorem proved in \cite{GL} for close families of sets is obtained here. This section ends with an outline of the strategy of the proof of our main theorem. The key steps are then carried out in Sections~\ref{low} and~\ref{high}. The main theorem is proved in Section \ref{MaxFam}, where we also discuss partial results concerning ``maximal families'' of homogeneous polynomials. 
Further, some related open questions are stated here and a remark mentioning briefly some of the applications of our main theorem is also included.  
%
%

\section{Preliminaries}

In this section we collect some preliminary notions and results, which will be needed later. These include a known answer to the affine analogue of the question posed at the beginning of this paper. As an application, we will settle the case when the polynomials have a linear factor
in common. 

Fix positive integers $r,d,m$ and a finite field $\Fq$ with $q$ elements. 
As in the Introduction, let $S:=\Fq[x_0,x_1, \dots , x_m]$ and for any $j\ge 0$, 
denote by $S_j$ or by $\Fq[x_0,x_1, \dots , x_m]_j$ the space of homogeneous polynomials in $S$ of (total) degree $j$. Note that $S_j$ is a $\Fq$-vector space of dimension $\binom{m+j}{j}$. 
With this in view, we will assume that $r\le \binom{m+d}{d}$. 
The notation $p_k$ (for $k\in \Z$) and $T_r(d.m)$ defined in the Introduction will be used frequently throughout this paper. 

\subsection{Initial Cases}
\label{subsec1}
It is easy to see that the TBC holds in the affirmative if $d=1$ or $m=1$. Indeed, if $d=1$, then by linear algebra, the number of common zeros in $\PP^m(\Fq)$ of  $r$ linearly independent homogeneous linear polynomials in $S$ is $p_{m-r}$, 
and on the other hand, 
$T_r(1,m) = p_{m-2r} + 1 \cdot (  p_{m-r} - p_{m-2r} ) = p_{m-r}$ as well. Likewise, 
if $m=1$,
then $(d-r+1, r-1)$ is the $r^{\rm th}$ ordered pair, in lexicographic descending order, among the pairs of nonnegative integers whose sum is $d$, and thus $T_r(d,1) = p_{-1} + 
(d-r+1)\left(p_{0} - p_{-1}\right) = d-r+1$. Now suppose $d \le q$. To  see that $d-r+1$ is indeed the maximum number of common zeros that 
$r$ linearly independent polynomials in $\Fq[x_0,x_1]_d$, say $F_1, \dots , F_r$,  have, 
one can proceed as follows. If $t$ is the number of common zeros of $F_1, \dots , F_r$, then there is a product, say $G$, of $t$ distinct polynomials in $\Fq[x_0,x_1]_1$ such that $F_i = GG_i$ for some 
$G_i\in \Fq[x_0,x_1]_{d-t}$ ($1\le i \le r$). Since $F_1, \dots , F_r$ are linearly independent, so are 
$G_1, \dots , G_r$, and hence $r\le \dim \Fq[x_0,x_1]_{d-t} = d-t+1$. Thus $t \le d - r + 1$. To see that the upper bound $d-r+1$ is attained, note that any $\mathbf{a} = (a_0:a_1)\in \PP^1(\Fq)$ gives rise to a homogeneous linear polynomial $L_{\mathbf{a}} = a_1 x_0 - a_0 x_1$ with $\mathbf{a}$ as its root, and conversely, any  homogeneous linear polynomial in $\Fq[x_0,x_1]$ has a unique root in $\PP^1(\Fq)$. Let $L_1, \dots , L_{q+1}$ be the homogeneous linear polynomials in $\Fq[x_0,x_1]$ corresponding to the $q+1$ distinct points of $\PP^1(\Fq)$.  For $i=1, \dots , r$, 
consider $F^*_i := L_1\cdots \widehat{L_i} \cdots L_{d+1}$, where $\widehat{L_i}$ indicates that ${L_i}$ is dropped from the product. Clearly, $F^*_1, \dots , F^*_r \in \Fq[x_0,x_1]_d$ and their common zeros are precisely the points of $\PP^1(\Fq)$ corresponding to the $d-r+1$ factors $L_{r+1}, \dots , L_{d+1}$. Moreover, if $F^*_1, \dots , F^*_r$ were linearly dependent, then one of them, say $F^*_i$, would be a $\Fq$-linear combination of others. But then the point of $\PP^1(\Fq)$ corresponding to $L_i$ would be a zero of  $F^*_i$, which is a contradiction. 

With this in view, we shall 
frequently assume that $d>1$ and $m>1$. In this case if 
for $1 \le i \le m+1$, we let $\mathbf{e}_i$ denote the $(m+1)$-tuple with $1$ in $i$th place and $0$ elsewhere, then 
the $r$th element in descending lexicographic order among the exponent vectors of monomials in $m+1$ variables of degree $d$ is precisely
$(d-1)  \mathbf{e}_1+\mathbf{e}_r$, provided $r \le  m+1$. Consequently,
$$
T_r (d,m)   = 
(d-1) q^{m-1} +  p_{m-2} +   q^{m-r}  \, \text{if $r\le m$ and } T_{m+1} (d,m)   = (d-1) q^{m-1} +  p_{m-2};
$$
in other words,
\begin{equation}
\label{expTr}
T_r (d,m)   =  (d-1) q^{m-1} +  p_{m-2} + \lfloor q^{m -r} \rfloor \ \text{if } \;  r \le m+1. 
\end{equation}

To end this subsection, we state for ease of reference the known answer to TBC in a nontrivial initial case of $r=1$. This result is also valid when $d=1$ or $m=1$. 

\begin{theorem}
\label{Serre}
Let $F$ be a nonzero homogeneous polynomial in $S$ of degree $d$ in $m+1$ variables. If $d \le q+1$, then $F$ can have at most $dq^{m-1} + p_{m-2}$ zeros in $\PP^m(\Fq)$. Moreover, if 
$d \le q+1$ and if $F$ has exactly 
$dq^{m-1} + p_{m-2}$ zeros in $\PP^m(\Fq)$, then $F$ is a product of $d$ distinct homogeneous linear polynomials, and the hyperplanes in $\PP^m$ corresponding to these linear factors have a codimension $2$ linear subspace in common. 
\end{theorem}

\begin{proof} 
For a proof of the first assertion, see Serre \cite{Se} or S{\o}rensen  \cite[Thm.~1]{So}  or  
\cite[Thm.~2.2]{DG}.
The second assertion is proved in  \cite{Se}. 
\end{proof}

\subsection{Projective varieties and Complete intersections} In this paper, by a projective variety we shall mean a projective algebraic set defined over $\Fq$. 
Thus varieties are not assumed irreducible, but if they happen to be irreducible, it will be stated explicitly. 
If $\mathcal{F}$ is a 
set of homogeneous polynomials in $S=\Fq[x_0,x_1, \dots , x_m]$, then we denote by $\V(\mathcal{F})$ the projective variety consisting of the common zeros in $\PP^m(\Fq)$ of polynomials in $\mathcal{F}$. 
If 
$\mathcal{F} = \{F_1, \dots , F_s\}$, we often write $\V(F_1, \dots, F_s)$ for $\V(\mathcal{F})$. A little more formally, if 
$\langle \mathcal{F} \rangle$ is the (homogeneous) ideal of $S$ generated by $\mathcal{F}$, then 
$\V(\mathcal{F})$ corresponds to the closed subscheme $\Proj(S/\langle \mathcal{F} \rangle)$ of 
$\PP^m = \Proj (S)$.

If $X$ is a projective variety (defined over $\Fq$), 
we denote by $\overline{X}$ the corresponding projective variety over the algebraic closure of $\Fq$. Given a projective variety $X$ in $\PP^m(\Fq)$, the notions of \emph{dimension} and \emph{degree} of $X$, denoted $\dim X$ and $\deg X$ respectively,  are understood in scheme-theoretic sense. These remain unchanged under a base change and could also be defined in terms of $\overline{X}$. If $X= V(F_1, \dots, F_s)$ for some homogeneous $F_1, \dots, F_s \in S$ and 
$\codim X := m - \dim X = s$, then $X$ is said to be a \emph{(scheme-theoretic) complete intersection}  in $\PP^m$; in this case the degrees $d_i = \deg F_i$, $i=1, \dots , s$, depend only on $X\hookrightarrow \PP^m$ and, moreover, we have $\deg X = d_1\cdots d_s$. Complete intersections of codimension $1$ in $\PP^m$ are precisely hypersurfaces, i.e., subvarieties of the form $\V(F)$ for some  homogeneous $F\in S$ of positive degree. 
The following simple observation will be useful to construct 
complete intersections other than hypersurfaces. 

\begin{lemma}
\label{complete}
Let $F_1, F_2$ be nonconstant homogeneous polynomials in $S$ having no nonconstant common factor. 
Then $\V(F_1, F_2)$ is a complete intersection of codimension $2$ in $\PP^m(\Fq)$ and, moreover, the degree of $\V(F_1, F_2)$ is $(\deg F_1)( \deg F_2)$. 
\end{lemma}

\begin{proof}
If $\p$ is a minimal prime ideal of the ideal $\langle F_1, F_2\rangle$ of $S$ generated by $F_1, F_2$, then by Krull's principal ideal theorem, the height of $\p$ is $\le 2$. If 
it were $<2$, then $\p$, being a height $1$ prime ideal in a UFD, would be principal, say $\langle F \rangle$, for some nonconstant $F\in S$. But then 
$\langle F_1, F_2\rangle \subseteq \p = \langle F\rangle$ implies $F$ divides $F_1$ and $F_2$, which  
is a contradiction. 
It follows that $\dim \V(F_1, F_2) = m-2$, as desired. The assertion about 
$\deg \V(F_1, F_2)$ follows from general facts about complete intersections. 
\end{proof}

The following basic bound for the number of $\Fq$-rational points of a projective variety 
over $\Fq$ is due to Lachaud, and a proof can be found in \cite[Prop. 12.1]{GL2}, except that the hypothesis of equidimensionality must be added. For alternative proofs one may refer to \cite[Thm. 2.1]{LR} or \cite[Prop. 2.3]{DG2}.

\begin {theorem}\label{cardinality}
Let $X \subset \mathbb P^m $ be an equidimensional projective variety defined over $\Fq$ of degree $\delta$ and dimension $n$. Then
$$|X (\Fq)|\leq \delta p_n.$$
\end {theorem}

In this paper, we will apply Theorem \ref{cardinality} to complete intersections such as those in Lemma \ref{complete}, and we will tacitly use here the well-known fact that complete intersections are equidimensional. In fact, in the case of varieties such as $\V(F_1, F_2)$ as in Lemma \ref{complete}, the proof shows that every minimal prime of $\langle F_1, F_2\rangle$ has height~$2$ and hence every irreducible component of $\V(F_1, F_2)$ has dimension $m-2$. 

\subsection{Affine case}
As remarked in the Introduction, the affine analogue of the TBC has been settled by Heijnen and Pellikaan \cite{HP} working in the context of Reed-Muller codes. Their result will be needed in this paper, and we 
state it below. 
A self-contained account of 
its proof can also be found in \cite[Appendix A]{DT}. 
\begin{theorem} 
\label{HP}
Assume that $1 \le d< q$. 
Then the maximum number of 
zeros in $\AA^m(\Fq)$ 
of a system of $r$ linearly independent  polynomials in $\Fq[x_1, \dots , x_m]$ of degree at most $d$ 
is 
\begin{equation}
\label{Hrdm}
H_r(d, m) := q^m - \left(\displaystyle{1+\sum_{j= 1}^{m}} \alpha_{j} q ^{m-j} \right),
\end{equation}
where $(\alpha_1, \dots , \alpha_m)$ is the $r^{\rm th}$ tuple in the set 
$\Lambda(d,m)$ of 
$m$-tuples $(\beta_1, \dots , \beta_m)$ with coordinates from $\{0,1, \dots , q-1\}$ satisfying $\beta_1+ \dots + \beta_m \ge m(q-1)-d$,
 and where the $m$-tuples are arranged lexicographically in ascending order. In particular, if $r\le m+1$, 
then this maximum number is $(d-1)q^{m-1} + \lfloor q^{m-r} \rfloor$. 
\end{theorem}

\begin{proof}
The first assertion is a restatement of \cite[Thm. 5.10]{HP}. To see the last assertion, 
note that 
$\alpha^*:=(q-1-d, \, q-1, \dots , \, q-1)$ is the least element of $\Lambda(d,m)$ and for $1< r \le m$, the $r^{\rm th}$ element is obtained from $\alpha^*$ by changing the first coordinate  to $q-d$ and the $r^{\rm th}$ coordinate 
to $q-2$, whereas the $(m+1)^{\rm th}$ element is 
$(q-d, \, q-1, \dots, \, q-1)$; consequently, $H_r(d,m)$ simplifies to 
$(d-1)q^{m-1} + q^{m-r}$ if $1\le r\le m $ and to $(d-1)q^{m-1}$ if $r=m+1$. 
\end{proof}

As an application of the above result, we show how the Tsfasman-Boguslavsky bound $T_r(d,m)$ can be readily obtained for intersections of hypersurfaces in $\PP^m$ of degree $d$ having a hyperplane in common. 

\begin{lemma}
\label{ded}
Assume that $r \le m+1$ and $1< d \le q$. 
Let $F_1, \dots, F_r$ be linearly independent homogeneous polynomials in $S_d$ having a common 
linear factor. Then 
\begin{equation}
\label{DesiredBound}
| \V(F_1, \dots, F_r)| \le (d - 1) q^{m-1} + p_{m-2} + \lfloor q^{m-r} \rfloor .
\end{equation}
\end{lemma}

\begin{proof}
Suppose $H \in S$ is a common 
linear factor of $F_1, \dots, F_r$. Then $H$ is necessarily  homogeneous and 
we may assume without loss of generality that $H = x_0$. Thus $x_0 \mid F_i$ for all $i = 1, \dots, r$. 
Write $f_i (x_1,  x_2, \dots, x_m) = F_i(1, x_1, \dots, x_m)$ for $i = 1, \dots, r$ and let $X'$
denote the set of common zeros in $\AA^m(\Fq)$ of the polynomials $f_1, \dots , f_r\in \Fq[x_1, \dots , x_m]$. 
Note that $X' = \V(F_1, \dots, F_r) \cap \{x_0 = 1 \}$ and so  
$$
 \V(F_1, \dots, F_r) = X'\cup X'' 
\quad \text{where} \quad X'' :=  \V(F_1, \dots, F_r) \cap \{x_0 = 0 \} = \V(x_0), 
$$
Since $F_1, \dots, F_r$ are linearly independent, so are $f_1, \dots, f_r$. Also  $\deg f_i \le d-1< q$ for each $i=1, \dots , r$. 
By Theorem \ref{HP},
 $|X'| \le (d-2)q^{m-1} + \lfloor q^{m-r} \rfloor$. It follows that 
\begin{equation*}
\begin{split}
| \V(F_1, \dots, F_r)|  = |X'| + |X''|  
& \le (d-2)q^{m-1} + \lfloor q^{m-r} \rfloor + p_{m-1}. 
\end{split}
\end{equation*}
This yields \eqref{DesiredBound}. 
\end{proof}

\section{Coprime Close Families} 
\label{comb}
Motivated by the notion of a ``close family of sets" 
introduced and studied in \cite{GL}, we consider an analogous notion for finite families of homogeneous polynomials of the same degree. We will be particularly interested when the polynomials in this family are relatively prime. 
In what follows, the fact that $S=\Fq[x_0,x_1, \dots , x_m]$ is a unique factorization domain (UFD) will be tacitly used; in particular, note that any finite collection of polynomials in $S$ have a gcd (= greatest common divisor) and it is unique up to multiplication by a nonzero constant, i.e., an element of $\Fq^*$. 
Thus it makes sense to talk about the degree of ``the'' gcd of finitely many polynomials. For $G_1, \dots , G_r\in S$, we shall often write $\gcd(G_1, \dots, G_r) = 1$ to mean that $G_1, \dots, G_r$ are relatively prime, i.e., they have no nonconstant common factor.  We will also tacitly use the elementary and well-known fact that factors of a homogeneous polynomial in $S$ are necessarily homogeneous. 
\begin{definition}\label{close}
Let $k$ be a positive integer and $\mathcal{G}_r = \{G_1, \dots, G_r\}$ be a subset of $S$ 
consisting of $r$ linearly independent homogeneous polynomials of degree $k$. We say that $\mathcal{G}_r$ is \emph{close} if $\deg \gcd (G_i, G_j) = k-1$ for all $i,j=1, \dots , r$ with $i\ne j$. Also we say that $\mathcal{G}_r$ is \emph{coprime close} if 
it is close and if  $\gcd(G_1, \dots, G_r) = 1$. 
\end{definition}

The original definition in \cite{GL} of a close family was in the context of subsets of cardinality $k$ of the set $[n]:=\{1, \dots , n\}$, where $n,k$ are positive integers with $k \le n$. In the same way, for an arbitrary set $N$ of cardinality $n$, upon letting $I_k(N)$ denote the set of all subsets of $N$ of cardinality $k$, we define a family $\Lambda \subseteq I_k(N)$ 
to be \emph{close} if 
$|A\cap B| = k-1$ for all $A, B \in \Lambda$ with $A\ne B$.  We state below a useful consequence of the 
Structure Theorem for Close Families proved in \cite{GL}. 

\begin{proposition}
\label{pro:structure}
Let $k,n$ be positive integers with $k\le n$, and let $N$ be a finite set with $n$ elements. Suppose 
$\Lambda \subseteq I_k(N)$ 
is close and $|\Lambda| =r \ge 1$. Then 
$$
\left|\displaystyle{\bigcap_{A \in \Lambda}}A\right| =k-1 \ \ {\rm or} \  \ k-r+1.
$$ 
Moreover,  if 
$1< k < n$ and if the intersection of all $A \in \Lambda$ is empty, then there exist distinct elements $\nu_1, \dots , \nu_r$ in $N$ such that
$$
\Lambda = \big\{ \{\nu_1, \dots, \check{\nu_i}, \dots , \nu_r\} : i=1, \dots , r \big\},
$$
where $ \check{\nu_i}$ indicates that $\nu_i$ is deleted. 
\end{proposition}

\begin{proof} 
If $r=1$, then there is nothing to prove. Suppose $r\ge 2$. Note that the proof of the 
Structure Theorem for Close Families \cite[Thm. 4.2]{GL}, the notions used therein from \cite[Defn. 4.1]{GL} and the observations in \cite[Remark 4.1]{GL} carry over \emph{verbatim} if $[n]$ is replaced by $N$. Now the desired result is an immediate consequence of Theorem~4.2 and Remark 4.1 of \cite{GL}. 
\end{proof}

In our setting of coprime close families of homogeneous polynomials, the result takes the following form. 
Recall that $r$ always denotes a positive integer. 

\begin{theorem}\label{structure}
Let $k$ be a positive integer and $\mathcal{G}_r=\{G_1, \dots , G_r\}$ be a coprime close family of $r$ linearly independent 
polynomials in $S_k$. Then  $k = 1$ or $k = r-1$. Moreover, if $k>1$, then there exist homogeneous linear polynomials $H_1, \dots , H_r\in S$ such that no two among $H_1, \dots , H_r$ differ by a nonzero constant, and moreover $G_i= H_1\cdots \check{H}_i \cdots H_r$, where $\check{H}_i$ indicates that the factor $H_i$ is omitted.
\end{theorem}

\begin{proof}
If $k=1$, there is nothing to prove. 
Suppose $k\ge 2$. Observe the following. 
\begin{enumerate}
\item[(i)] No polynomial in $\mathcal{G}_r$ has an irreducible factor of degree $\ge 2$. 

\item[(ii)] No polynomial in $\mathcal{G}_r$ has a repeated linear factor, i.e., $H^2 \nmid G_i$ for all $i=1, \dots , r$ and $H\in S_1$. 
\end{enumerate}
To see (i), suppose $Q \mid G_i$ for some $i\in \{1, \dots , r\}$ and $Q \in S$, where $Q$ is  irreducible of degree $\ge 2$. 
Since $\deg G_j = k = \deg G_i$ and $\deg \gcd (G_i, G_j) = k-1$ for all $j=1, \dots , r$ with $j\ne i$, it follows that $Q\mid G_j$ for all $j=1, \dots , r$. But this contradicts the assumption that $\gcd(G_1, \dots, G_r) = 1$. Likewise, to see (ii) suppose $H^2 \mid G_i$ for some $i\in \{1, \dots , r\}$ and $H \in S_1$.  Then $H\mid G_j$ for all $j=1, \dots , r$, again contradicting 
$\gcd(G_1, \dots, G_r) = 1$. From (i) and (ii), we deduce that each $G_i$ is a product of $k$ homogeneous linear factors, which are distinct in the sense that no two of them differ by a nonzero constant. 
Let us define two elements of $S$ to be equivalent if they differ by a nonzero constant. 
This induces an equivalence relation on the set $S_1\setminus\{0\}$ of nonzero homogeneous linear polynomials; let $N$ denote the set of equivalence classes. Note that $N$ is a finite set of cardinality $n:=p_m$. For each $G_i \in \mathcal{G}_r$, let $A_i$ denote the set of equivalence classes of homogeneous linear factors of $G_i$. 
Then $\Lambda:=\{A_1, \dots , A_r\}$ is a close family in $I_k(N)$. 
Moreover, since $\gcd(G_1, \dots, G_r) = 1$, we must have 
$|A_1\cap \cdots \cap A_r| = 0$. 
Now the desired result follows readily from 
Proposition \ref{pro:structure}.
\end{proof}
%

We will now outline a general strategy to prove the TBC when $1 < r\le m+1$ and $1<d<q-1$. The notations introduced here will be used in the next two sections. 
Let $F_1, \dots, F_r$ be linearly independent homogeneous polynomials in $S_d$. Fix a gcd $G$ of $F_1, \dots , F_r$ and let $G_1, \dots , G_r\in S$ be such that $F_i = G G_i$ for $i=1, \dots , r$. Also fix a gcd, say $F_{ij}$, of $F_i$ and $F_j$ as well as a gcd, say $G_{ij}$, of $G_i$ and $G_j$ for all $i,j=1, \dots , r$ with $i\ne j$. Note that $G, G_i, F_{ij}$ and $G_{ij}$ are homogeneous. Let 
$$
b:= \deg G \quad \text{and} \quad\quad b_{ij}:= \deg F_{ij} \quad \text{for $i,j=1, \dots , r$ with $i\ne j$. }
$$
Evidently $\deg G_i = d-b$ for all $i=1, \dots , r$ and $\deg G_{ij} = b_{ij} - b$  for all $i,j=1, \dots , r$ with $i\ne j$. We will refer to $b_{ij}$ as the \emph{correlation factor} between $F_i$ and $F_j$. Since 
$F_1, \dots, F_r$ are linearly independent, we see that $G_1, \dots , G_r$ are linearly independent and $0\le b_{ij} \le d-1$ for all $i,j=1, \dots , r$ with $i\ne j$. Also it is clear that $\gcd(G_1, \dots , G_r) = 1$ . The proof will be divided into three cases as follows. 
\begin{description}
\item[Case 1] $b_{ij} = 0$ for some $i,j\in\{1, \dots , r\}$ with $i \ne j$. 

\item[Case 2] $0 < b_{ij} < d-1$ for some $i,j\in\{1, \dots , r\}$ with $i \ne j$.

\item[Case 3] $b_{ij} = d-1$ for all $i,j\in\{1, \dots , r\}$ with $i \ne j$.
\end{description}
The first two cases will be referred to as that of low correlation and will be dealt with in Section \ref{low} below. In Case~3, we see that $\{G_1, \dots , G_r\}$ is a coprime close family in $S_k$ where $k:=d-b$. 
Hence in view of Theorem \ref{structure}, this case divides itself into exactly two subcases: (i) $b= d-1$, and 
(ii) $b=d-r+1$. These two will be considered in Section \ref{high}. The goal in each case is to prove an inequality such as \eqref{DesiredBound}. In the case of low correlation, we will in fact obtain a better bound. 

\section{The Case of Low Correlation}
\label{low}

The first two cases in the strategy outlined at the end of Section \ref{comb} will be considered in the following two lemmas. 
It will be seen that in each of them, we obtain an inequality better than the desired one, namely, \eqref{DesiredBound}. In particular,  the 
Tsfasman-Boguslavsky bound $T_r(d,m)$ is not attained in these cases. The arguments in this section are reminiscent of those in the proof of Theorem 2 in Boguslavsky \cite{Bog}. 

\begin{lemma} 
\label{case1}
Assume that $ r > 1$ and $1<d< q-1$. Let 
$F_1, \dots, F_r$ be linearly independent 
polynomials in $S_d$ such that 
$\deg \gcd (F_i,F_j) =0$ for some $i,j=1, \dots , r$ with $i \ne j$. 
Then
$$
| \V(F_1, \dots, F_r)| < (d - 1) q^{m-1} + p_{m-2}. 
$$
\end{lemma}

\begin{proof}
Let us assume, 
without loss of generality, that $b_{12} = 0$, i.e., $F_1, F_2$ do not have a nonconstant common factor. Now by  
Lemma  \ref{complete}, 
$V(F_1, F_2)$ is a complete intersection and hence by 
Theorem~\ref{cardinality}, 
\begin{equation*}
\begin{split}
|\V(F_1, F_2)| &\le  d^2 p_{m-2}\\
		& = (d-1)(d+1) p_{m-2} + p_{m-2}\\
		& \le (d-1)(q-1) p_{m-2} + p_{m-2}  \qquad  \text{[since $d< q-1$]}\\ 
		& = (d-1) (q^{m-1} - 1) + p_{m-2}\\
		& <(d-1)q^{m-1} + p_{m-2}   \qquad \qquad \quad \text{[since $d > 1$].}\\
\end{split}
\end{equation*}
As a consequence, $|\V(F_1, F_2, \dots, F_r)| \le |\V(F_1, F_2)| < (d-1)q^{m-1} + p_{m-2}$. 
\end{proof}


\begin{lemma} 
\label{case2}
Assume that $ r > 1$ and $1<d< q-1$. Let  $F_1, \dots, F_r$ be linearly independent 
polynomials in $S_d$ such that
$0 < \deg \gcd (F_i,F_j) < d- 1 $ for some $i,j=1, \dots , r$ with $i \ne j$. 
Then
$$
| \V(F_1, \dots, F_r)| < (d - 1) q^{m-1} + p_{m-2}. 
$$
\end{lemma}

\begin{proof}
Let us assume, 
without loss of generality, that 
$0<b_{12}<d-1$. Fix a gcd $F_{12}$ of $F_1$ and $F_2$ and let $Q_1, Q_2\in S$ be such that 
$F_i = F_{12}Q_i$ for $i=1,2$. Note that $Q_1$ and $Q_2$ are 
coprime and both are nonconstant homogeneous polynomials of degree $d-b_{12}$. Let 
$$
X' = \V(F_1, F_2),  \quad Y' = \V(F_{12}) \quad \text{and} \quad X'' = \V(Q_1, Q_2). 
$$
In view of Lemma \ref{complete}, $X''$ is a complete intersection of dimension $m-2$ and degree
$(d - b_{12})^2$ and consequently by Theorem \ref{cardinality}, 
 $|X ''| \le (d - b_{12})^2 p_{m-2}$. On the other hand, Theorem \ref{Serre} applies to $Y'$ and so 
$|Y'| \le b_{12}q^{m-1} + p_{m-2}$. 
It follows that 
$$
|X'| \le |Y'| + |X''| \le b_{12} q^{m-1} + p_{m-2} + (d-b_{12})^2 p_{m-2}.
$$
We shall now estimate the difference between $|X'|$ and $T_2(d,m)$. 
\begin{equation*}
\begin{split}
& |X'| - (d-1)q^{m-1} - p_{m-2} -q^{m-2} \\
					& \le (b_{12}-d+1) q^{m-1} + (d-b_{12})^2 p_{m-2} - q^{m-2} \\
					& = -\frac{1}{q-1} \left[ (d-b_{12}-1) q^{m-1} (q-1)  - (d-b_{12})^2 (q^{m-1} -1)
+ q^{m-1} - q^{m-2} \right]\\
					& =-\frac{1}{q-1}[q^{m-1} (q-1) (d-b_{12}-1) -q^{m-1} \{(d-b_{12})^2 - 1\} + (d-b_{12})^2 - q^{m-2}]\\ 
					& = -\frac{1}{q-1} [q^{m-1} (d-b_{12}-1) (q - d + b_{12} -2) + (d-b_{12})^2 - q^{m-2}]
\end{split}
\end{equation*}
Since $0<b_{12}<(d-1)$, we have $d-b_{12}-1\ge 1$. Also $q-1>d$. 
Consequently, 
 $q - d + b_{12} - 2 \ge 1$. Thus, 
\begin{eqnarray*}
\begin{split}
 & \  |X '| - (d-1)q^{m-1} - p_{m-2} -q^{m-2}  \\
& \le    -\frac{1}{q-1} [q^{m-1} (d-b_{12}-1) (q - d + b_{12} -2) + (d-b_{12})^2 - q^{m-2}] \\
&<   -\frac{1}{q-1} [q^{m-1} - q^{m-2}] 
= - q^{m-2}.
\end{split}
\end {eqnarray*}
It follows that 
$$
|X'| - (d-1)q^{m-1} - p_{m-2}  < - q^{m-2} + q^{m-2}  = 0.$$
Thus, $ |X|\le |X'| < (d-1)q^{m-1} +  p_{m-2}$, as desired.  
\end{proof}

\section{The Case of High Correlation}
\label{high}

As usual, we will denote by $\hpm$ the dual projective space consisting of all hyperplanes in $\PP^m$; in other words, $\hpm$ is the collection of $\V (H)$ as $H$ varies over nonzero homogeneous linear polynomials  in $S:=\Fq[x_0,x_1, \dots , x_m]$. We begin with a somewhat general proposition about 
intersections of hyperplanes in projective spaces, which will be useful 
later. Although we continue to assume that the base field is $\Fq$, this result and its proof is valid if $\Fq$ is replaced by an arbitrary field. 

\begin{proposition}
\label{Hplanes}
Assume that $1\le r \le m+1$. Let $H_1, \dots , H_r\in S_1$ be linearly independent homogeneous linear polynomials and let $\Pi_i:= \V(H_i)$ denote the hyperplane in $\PP^m$ defined by $H_i$ for $i=1, \dots , r$. Let $L:=\V(H_1, \dots , H_r)$ be the linear subvariety of $\PP^m$ defined by $H_1, \dots , H_r$ and $P$ be a point of $\PP^m$ such that $P\not\in L$. Then for any $\Pi\in \hpm$ passing through $P$, upon letting $L_{\Pi}:= L \cap \Pi$, we have 
$$
\codim_{\Pi} L_{\Pi} \;  = \; r-1 \; \text{ or } \; r.
$$
Moreover, if $H\in S_1$ is such that $\Pi= \V(H)$, then 
\begin{eqnarray*}
\codim_{\Pi} L_{\Pi} = r-1 &\Longleftrightarrow& \text{the restrictions } {H_1|}_{\Pi},  \dots, {H_r|}_\Pi \text{ are linearly dependent}  \\
&\Longleftrightarrow&  H = \sum_{i=1}^r \lambda_i H_i \text{ for some $\lambda_1, \dots, \lambda_r\in \Fq$,  not all zero}.
\end{eqnarray*}
\end{proposition}

\begin{proof}
Fix $P\in \PP^m \setminus L$ and let $0\ne H\in S_1$ and $\Pi = \V(H) \in \hpm$ be such that  $P\in \Pi$. 
By a linear change of coordinates, we may assume that $H = x_m$. Thus $\Pi$ can be nicely identified with $\PP^{m-1}$. Let  $\widetilde{H}_i(x_0, \dots , x_{m-1}):= H_i(x_0, \dots , x_{m-1}, 0)$ be the restriction of $H_i$ to $\Pi$ and let 
 $c_i \in \Fq$ be such that 
$H_i = \widetilde{H}_i + c_ix_m$ for $i =1, \dots , r$. 
Now $L_{\Pi}:= L \cap \Pi$ is the linear subvariety in $\PP^{m-1}$ defined by the vanishing of $\widetilde{H}_1, \dots , \widetilde{H}_r$. If $\widetilde{H}_1, \dots , \widetilde{H}_r$ are linearly independent, then it is clear that $\codim_{\Pi} L_{\Pi}   = r$. On the other hand, suppose $\widetilde{H}_1, \dots , \widetilde{H}_r$ are linearly dependent. Then there exist $\lambda_1, \dots, \lambda_r\in \Fq$,  not all zero, such that 
$$
\sum_{i=1}^r \lambda_i \widetilde{H}_i = 0 \quad \text{and hence} \quad 
\sum_{i=1}^r \lambda_i H_i = c \, x_m, \;   \text{ where } \; c: = \sum_{i=1}^r \lambda_i c_i.
$$
Since $H_1, \dots , H_r$ are linearly independent, we must have $c\ne 0$ and hence $L$ is unchanged if we replace one of the $H_i$'s by $x_m$. Suppose, without loss of generality, $H_1 = x_m$. Now $L_{\Pi}$ is defined by the vanishing of $\widetilde{H}_2, \dots , \widetilde{H}_r$. Moreover,   $\widetilde{H}_2, \dots , \widetilde{H}_r$ are linearly independent. It follows that $\codim_{\Pi} L_{\Pi}   = r-1$. This proves all the assertions in the lemma. 
\end{proof}

\begin{corollary}
\label{CorHplanes}
Assume that $1\le r \le m+1$. Let $H_1, \dots , H_r\in S_1$ be linearly independent  and let $L:=\V(H_1, \dots , H_r)$ and $P\in \PP^m \setminus L$. Then 
$$
\left| \left\{ \Pi \in \hpm : P \in \Pi \text{ and } \codim_{\Pi} L_{\Pi}   = r-1\right\}\right| = p_{r-2},
$$
where as in Proposition \ref{Hplanes}, $L_{\Pi}:= L \cap \Pi$ for any $\Pi \in \hpm$. 
\end{corollary}

\begin{proof}
Since $P\in \PP^m \setminus L$, the evaluations $H_1(P), \dots , H_r(P)$ are not all zero. By Proposition \ref{Hplanes}, the set 
$$
 \left\{ \Pi \in \hpm : P \in \Pi \text{ and } \codim_{\Pi} L_{\Pi}   = r-1\right\}
$$ 
can be identified with the set  $ \left\{ (\lambda_1: \dots : \lambda_r) \in \PP^{r-1}(\Fq) : \sum_{i=1}^r \lambda_i H_i(P) = 0\right\}$, and the cardinality of the latter is clearly $p_{r-2}$. 
\end{proof}


Next lemma 
corresponds to the first subcase of Case 3 in the general strategy outlined at the end of Section \ref{comb}, but with 
the case covered by Lemma \ref{ded} excluded. 

\begin{lemma}
\label{Case3-1}
Assume that $1<d \le q$ and $1\le r \le m+1$. Let 
$F_1, \dots, F_r$ be linearly independent 
polynomials in $S_d$ and let $G$ be a gcd of $F_1, \dots , F_r$. 
If $\deg G=d-1$ and if $G$ has no linear factor, then
\begin{equation}
\label{StrictIneq}
| \V(F_1, \dots, F_r)| < (d - 1) q^{m-1} + p_{m-2}+  \lfloor q^{m-r} \rfloor .
\end{equation}
\end{lemma}

\begin{proof} 
We use induction on $m$ to show that \eqref{StrictIneq} holds for every positive integer $r\le m+1$ and any $F_1, \dots , F_r\in S_d$ satisfying the hypothesis of the lemma. 
In the remainder of the proof, we will use the following notation. 
With $F_1, \dots, F_r$ and $G$ as in the statement of the lemma, we let $H_1, \dots, H_r$ be 
linear homogeneous polynomials in $S$ such that $F_i=GH_i$ for $i=1, \dots , r$. 
Write $X:= \V(F_1, \dots, F_r)$, $Y:= \V(G)$ and $L = \V(H_1, \dots, H_r)$. Clearly $X = Y\cup L$. Note that since $F_1, \dots, F_r$ be linearly independent, so are $H_1, \dots, H_r$, and therefore $|L|=p_{m-r}$.

First, suppose $m=1$.  By our assumption $G(x_0,x_1)$ 
has no linear factor and hence $Y$ is empty and so $X= L$. 
It is now easy to see that  \eqref{StrictIneq} holds in this case. 
%

Next suppose $m>1$ and the result holds for smaller values of $m$. Fix a positive integer $r\le m+1$ and any $F_1, \dots , F_r\in S_d$ as in the statement of the lemma. Let $G, H_i, X, Y$ and $L$ be as above. Note that the case $r=1$ can not arise since $\deg G = d-1 < \deg F_1$.  Also note that if $r=m+1$, then $L$ is empty and $X=Y$; hence Theorem \ref{Serre} implies  \eqref{StrictIneq} in this case since $G$ has degree $d-1$ and has no linear factor. Thus we will assume that $2\le r \le m$. 
Observe that if 
$Y\subseteq L$, then 
$$
|X| = |L| = p_{m-r} <  p_{m-1} +  \lfloor q^{m-r} \rfloor \le (d - 1) q^{m-1} + p_{m-2}+  \lfloor q^{m-r} \rfloor,
$$
as desired. Thus we now assume that $Y\not\subseteq L$. Fix some $Q \in Y \setminus L$. Consider
$$
\mathscr{X} := 
\left\{(\Pi, P) \in \hpm \times \mathbb{P}^m : Q \in \Pi, \; P \in \Pi \cap X \text{ and } P \neq Q\right\}
$$
and let us count it in two ways. First, for a fixed $P\in X\setminus\{Q\}$, there are exactly $p_{m-2}$ hyperplanes $\Pi \in \hpm$ passing through the two distinct points $P$ and $Q$. Hence 
\begin{equation}
\label{eqX1}
|\mathscr{X}| = (|X| - 1)p_{m-2}.
\end{equation}
On the other hand, there are a total of $p_{m-1}$ hyperplanes  $\Pi \in \hpm$ that contain $Q$ and for each of them, a point $P\in \PP^m$ is such that $(\Pi,P)\in \mathscr{X}$ if and only if $P\in  (\Pi \cap X) \setminus\{Q\}$. Moreover, by Proposition \ref{Hplanes}, for any $\Pi \in \hpm$, the codimension of $L_{\Pi}:=L\cap \Pi$ in $\Pi$ is either $r-1$ or $r$.  Thus 
\begin{equation}
\label{eqX2}
|\mathscr{X}| =  \sum_{\substack{ \Pi \in \hpm \\ Q\in \Pi, \; \codim_{\Pi} L_{\Pi} = r-1}} (|\Pi \cap X| - 1)  \; + \sum_{\substack{\Pi \in \hpm \\  Q\in \Pi, \; \codim_{\Pi} L_{\Pi} = r}} (|\Pi \cap X| - 1).
\end{equation}
Denote the first and second sums on the right hand side of \eqref{eqX2} by $\Sigma_{r-1}$ and $\Sigma_r$ respectively. 
Since  $2 \le r \le m$,  using Corollary \ref{CorHplanes}, Proposition \ref{Hplanes} and the induction hypothesis together with Lemma \ref{ded} (applied to the restrictions of $F_1, \dots , F_r$ to $\Pi \simeq \PP^{m-1}$), we see that 
\begin{equation}
\label{eqS1}
\Sigma_{r-1} \le p_{r-2} \left( (d-1)q^{m-2} + p_{m-3} + q^{(m-1)-(r-1)}  - 1 \right). 
\end{equation}
Likewise, if $2\le r < m$, then using Corollary \ref{CorHplanes}, Proposition \ref{Hplanes} and the induction hypothesis together with Lemma \ref{ded}, we see that 
\begin{equation}
\label{eqS2}
\Sigma_{r} \le \left( p_{m-1} - p_{r-2} \right) \left( (d-1)q^{m-2} + p_{m-3} + q^{(m-1)-r}  - 1 \right). 
\end{equation}
In case $r=m$, for any $\Pi \in \hpm$ such that $\codim_{\Pi} L_{\Pi} = r$, the intersection $\Pi \cap L$ is empty and hence $\Pi\cap X = \Pi \cap Y$; consequently,  Theorem \ref{Serre} can be applied to deduce that 
$|\Pi\cap X| \le (d-1)q^{m-2} + p_{m-3}$. Thus \eqref{eqS2} holds in this case as well. Now adding the 
upper bounds in \eqref{eqS1} and \eqref{eqS2}, we see after some simplification 
that 
$$
|\mathscr{X}| \le p_{m-1} (d-1)q^{m-2} + p_{m-1}p_{m-3} +p_{r-2}(q^{m-r} - q^{m-r-1}) + p_{m-1}q^{m-r -1} - p_{m-1}.
$$
Putting $p_{m-1} = qp_{m-2} + 1$ in the first, second, and fourth summands of 
the right hand side of the above inequality, and then comparing with  \eqref{eqX1}, we obtain 
$$
|X| \le  1 +  (d-1)q^{m-1} + qp_{m-3}   + q^{m-r} + \frac{A}{p_{m-2}} = (d-1)q^{m-1} + p_{m-2}   + q^{m-r} + \frac{A}{p_{m-2}} ,
$$
where we have temporarily put
$$
A:= (d - 1)q^{m-2} + p_{m-3} +p_{r-2}(q^{m-r} - q^{m-r-1}) +q^{m-r -1} - p_{m-1} .
$$
To complete the proof, it suffices to show that $A<0$. To this end, observe that 
\begin{equation*}
\begin{split}
A &= (d-1)q^{m-2}  + q^{m-r-1}(q^{r-1} - 1) +q^{m-r-1} + p_{m-3} - p_{m-1}  \\
&= (d-1)q^{m-2} + q^{m-2} - q^{m - r- 1} + q^{m-r-1}  - q^{m-2} - q^{m-1}\\
&= (d-1 - q)q^{m-2} . 
\end{split}
\end{equation*}
Since $d\le q$, we see that $A < 0$.
\end{proof}

We now deal with the second subcase of Case 3 in the general strategy outlined at the end of Section \ref{comb}, but with 
the cases covered by Lemmas \ref{ded} 
and \ref{Case3-1} excluded. 

\begin{lemma}
\label{Case3-2}
Assume that $1<d < q$ and $1\le r \le m+1$. Let 
$F_1, \dots, F_r$ be linearly independent 
polynomials in $S_d$ and let $G$ be a GCD of $F_1, \dots , F_r$. 
Suppose $\deg \gcd (F_i, F_j) = d-1$ for all $i,j=1, \dots , r$ with $i\ne j$ and $\deg G < d-1$. Also suppose  
$F_1, \dots , F_r$ have no common linear factor. Then
\begin{equation}
\label{StrictIneq2}
| \V(F_1, \dots, F_r)| < (d - 1) q^{m-1} + p_{m-2}+  \lfloor q^{m-r} \rfloor .
\end{equation}
\end{lemma}

\begin{proof} 
Let $G_1, \dots , G_r\in S$ be such that $F_i=GG_i$ for $i=1, \dots , r$. Note that  $\{G_1, \dots , G_r\}$  
is a coprime close family of linearly independent homogeneous polynomials in $S$ of degree $k:= d- \deg G$. Also note that $k > 1$ by the hypothesis on $\deg G$. Thus by Theorem \ref{structure}, 
$r = k+1 \ge 3$ (so that $m\ge 2$) and 
there exist $H_1, \dots , H_r\in S_1$, no two $H_i$'s differing  by a nonzero constant, such that 
$$
G_i= H_1\cdots \check{H}_i \cdots H_r \quad \text{and} \quad  F_i = G  H_1\cdots \check{H}_i \cdots H_r
\quad \text{ for } i=1, \dots , r,
$$ 
where $\check{H}_i$ indicates that the factor $H_i$ is omitted. 
Note that $H_1 \mid F_i$ for $2 \le i \le r$, whereas $H_1 \nmid F_1$ since $F_1, \dots , F_r$ have no common linear factor.  By a linear change of coordinates, we may assume that $H_1 = x_0$. Now let 
$$
X := \V(F_1, \dots, F_r), \quad X_1 : = X \cap\V(x_0) \quad \text{and}  \quad X_2 := X \cap \left( \PP^m\setminus \V(x_0) \right) . 
$$ 
Clearly, $|X| = |X_1| + |X_2|$. Moreover, $X_1$ corresponds to a projective hypersurface in $\PP^{m-1}$ given by the vanishing of the nonzero homogeneous polynomial $F(0, x_1, \dots x_{m})$ of degree $d$.
Hence by Theorem \ref{Serre}, 
$$
|X_1| \le dq^{m-2} + p_{m-3}.
$$
On the other hand, $X_2$ is in bijection with the affine variety in $\AA^m$ defined by the vanishing of $f_1, f_2, \dots , f_r$, where $f_i(x_1, \dots , x_m):= F_i(1, x_1, \dots , x_m)$ for $i=1, \dots, r$. 
In particular, $X_2$ is a subset of the set of common zeros in $\AA^m(\Fq)$ of the $r-1$ polynomials 
$f_2, \dots , f_r$. Since each of $f_2, \dots , f_r$ has degree $\le d-1$, it follows from Theorem \ref{HP} that 
$$
|X_2| \le (d-2)q^{m-1} + q^{m-r+1}.
$$ 
Consequently, 
$$
|X| \le dq^{m-2} + p_{m-3} + (d-2)q^{m-1} + q^{m-r+1}.
$$
To complete the proof, it suffices to show that 
$$
(d-1)q^{m-1} + p_{m-2} + \lfloor q^{m-r} \rfloor >  dq^{m-2} + p_{m-3} + (d-2)q^{m-1} + q^{m-r+1}.
$$
To this end,  let us note that  for $1 \le r \le m$, the difference can be written as 
\begin{equation*}
\begin{split}
&  \big( (d-1)q^{m-1} + p_{m-2} + q^{m-r}  \big) - \big(  dq^{m-2} + p_{m-3} + (d-2)q^{m-1} + q^{m-r+1} \big) \\
& = q^{m-1} + q^{m-2}  - d q^{m-2} -q^{m-r} (q - 1)\\
& = q^{m-r} [(q - d + 1) q^{r-2} - (q-1)] \\
& \ge q^{m-r}(q^{r-2} - q +1)>0, 
\end{split}
\end{equation*}
where the last inequality holds since $r \ge 3$.
On the other hand for $r = m +1$, the difference is 
\begin{equation*}
\begin{split}
& \big( (d-1)q^{m-1}  + p_{m-2} \big) - \big( dq^{m-2} + p_{m-3} + (d-2)q^{m-1} + 1 \big)\\
& = q^{m-1} - (d-1)q^{m-2} -1 \\
& = q^{m-2} (q - d + 1) - 1 > 0, 
\end{split}
\end{equation*}
where the last inequality follows from the fact that $d<q$ and $m \ge 2$. 
\end{proof}

\begin{remark}
The above proof also shows that with the hypothesis on $r$ and $F_1, \dots , F_r$ as in Lemma \ref{Case3-2}, 
the weaker inequality
$$
| \V(F_1, \dots, F_r)| \le (d - 1) q^{m-1} + p_{m-2}+  \lfloor q^{m-r} \rfloor 
$$
holds under a somewhat more general assumption that $1< d \le q$. 
In fact, the only case where the proof does not yield the strict inequality \eqref{StrictIneq2} is 
$d=q$ and $m=2$. 
\end{remark}

\section{Maximal Families of Polynomials}
\label{MaxFam}
The results of the previous sections yield an upper bound on the number of common solutions of a system of $r$ linearly independent homogeneous polynomials in $\Fq[x_0, x_1, \dots , x_m]_d$ when $r\le m+1$ and $d < q-1$. The next lemma shows that this bound can be attained. 

\begin{lemma}
\label{lub}
Assume that $1\le d \le q+1$ and $1\le r \le m+1$. Then there exist~$r$ linearly independent homogeneous polynomials $F^*_1,  \dots, F^*_r \in S$  of degree $d$ such that 
$$
\left| \V(F^*_1,  \dots, F^*_r)\right| = (d-1)q^{m-1} + p_{m-2} + \lfloor q^{m-r}\rfloor.
$$
\end{lemma}

\begin{proof}
Since $d\le q+1$, we can choose $d-1$ distinct elements, say $\lambda_1, \dots, \lambda_{d-1}$, in $\Fq$. 
Consider the homogeneous polynomials $G^*$ and $F^*_1,  \dots, F^*_r $ 
defined by 
$$
G^* := (x_{m} - \lambda_1 x_0) \dots (x_m - \lambda_{d-1} x_0) \quad \text{and} \quad 
F^*_i := x_{i-1}G^* \quad \text{for } i = 1,  \dots, r.
$$ 
It is clear that $F^*_1,  \dots, F^*_r$ 
are linearly independent elements of $S_d$. 
Now let
$$
X = \V(F^*_1,  \dots, F^*_r), \quad Y = \V(G^*)  \quad \text{and} \quad  X':= \V(x_0, x_1, \dots, x_{r-1}). 
$$
Note that $X=Y \cup X'$ and so  $|X|=  |Y| + |X'| - |Y \cap X'|$. The points of $Y$ have homogeneous coordinates $(a_0:a_1:\dots : a_m)$ that fall into two disjoint classes: 
\begin{enumerate}
\item[(i)] $a_0 =1$, $a_m = \lambda_j$ for some $j \in \{1, \dots , d-1\}$ and $a_1, \dots , a_{m-1}\in \Fq$ arbitrary; 
\item[(ii)] $a_0=0=a_m$ and $(a_1: \dots : a_{m-1})\in \PP^{m-2}(\Fq)$ arbitrary. 
\end{enumerate}
Consequently, $|Y| = (d-1)q^{m-1} + p_{m-2}$. 
Also, $Y\cap X'=\V(x_0, x_1, \dots, x_{r-1}, x_m).$ It follows that $|X'| = p_{m-r}$ and $|Y\cap X'| = p_{m-r-1}$. Thus
$$
|X| =  (d-1)q^{m-1} + p_{m-2} + p_{m-r} - p_{m-r-1} = (d-1)q^{m-1} + p_{m-2} + \lfloor q^{m-r}\rfloor, 
$$
where we have used the fact that if $r=m+1$, then $p_{m-r}  = p_{m-r-1} = 0$. 
\end{proof}

\begin{remark}
\label{qplusone}
Suppose $d=q+1$ and $F^*_1,  \dots, F^*_r$ are the polynomials in $S_d$ as 
constructed in the above proof. 
Then $|\V(F_1^*)| = p_m$ and $F_1^*$  is precisely the polynomial $x_m^qx_0 - x_0^q x_m$. 
On the other hand, if  $2\le  r \le m+1$, then 
 $\left| \V(F^*_1,  \dots, F^*_r)\right| < p_m$. 
However, for any $r\le m+1$ and more generally, for any $r\le {{m+1}\choose{2}}$, it is easy to construct a family $H^*_1,  \dots, H^*_r$ of linearly independent polynomials in $S_{q+1}$ such that 
$|\V(H^*_1,  \dots, H^*_r)| = p_m$. Indeed, we can simply choose any $r$ distinct polynomials among the 
${{m+1}\choose{2}}$ \emph{Fermat polynomials}\footnote{The nomenclature \emph{Fermat polynomial} is motivated by Fermat's little theorem, which says that 
if $p$ is prime, then the polynomial $x^p - x$ vanishes at every point of $\FF_p =\Z/p\Z$; more generally, 
the polynomial $x^q-x$ vanishes at every point of $\Fq$.} $x_i^qx_j - x_j^q x_i$ for $0\le i< j \le m$. Showing linear independence is easy (e.g., if a linear combination equals zero, then setting all the variables except $x_i$ and $x_j$ to be zero, one finds that the coefficient of $x_i^qx_j - x_j^q x_i$ is necessarily zero) and each of these polynomials vanishes at every point of $\PP^m(\Fq)$. In fact, these 
${{m+1}\choose{2}}$  Fermat polynomials generate the vanishing ideal ${\mathscr{I}}$ of $\PP^m(\Fq)$; see, e.g., \cite{MR}. Further, as P. Beelen pointed out to us, 
if $d\ge q+1$, then there is $r_d \ge {{m+1}\choose{2}}$ and a family $H^*_1,  \dots, H^*_{r_d}$ of linearly independent polynomials in ${\mathscr{I}}_d := {\mathscr{I}} \cap S_d$ such that 
$|\V(H^*_1,  \dots, H^*_{r_d})| = p_m$. In fact, by 
\cite[Thm. 5.2]{MR}, 
$$
r_d = \dim {\mathscr{I}}_d = \sum_{j=2}^{m+1} (-1)^j {{m+1}\choose{j}} \sum_{i=0}^{j-2} {{d+(i+1)(q-1) - jq -m}\choose{d+(i+1)(q-1) - jq }}.
$$
\end{remark}

We are now ready to state and prove the main theorem of this paper. 

\begin{theorem}
\label{mainthm}
Assume that  $1\le d < q -1$ and $1\le r \le m+1$. Then the maximum number of 
zeros in $\PP^m(\Fq)$ that  a system of $r$ linearly independent homogeneous polynomials, each of degree $d$ in 
$S = \Fq[x_0,x_1, \dots , x_m]$,  
can have is given by the Tsfasman-Boguslavsky bound $T_r(d,m)$ given by \eqref{TBr} and 
more explicitly~by 
$$
T_r(d,m)= \begin{cases} p_{m-r} & \text{ if } \, d = 1 \text{ and } \, 1\le r \le m+1, \\
(d-1)q^{m-1} + p_{m-2} + \lfloor q^{m-r}\rfloor & \text{ if } \, d > 1 \text{ and } \, 1\le r \le m+1. \end{cases}
$$
Moreover if $d=1$, then the maximum 
$T_r(d,m)$ is always attained, whereas if $d>1$ and if 
 the maximum 
is attained
 by a family $\{F_1, \dots , F_r\}$ of $r$ linearly independent polynomials in $S_d$, then $F_1, \dots , F_r$ must have a common linear factor in $S$. 
\end{theorem}

\begin{proof}
The cases when $d=1$ or $r=1$ have 
already been discussed in \S \;\ref{subsec1}. Now suppose 
$1<d< q-1$ and $1 < r\le m+1$. Then $T_r(d,m)$ is given explicitly by \eqref{expTr}. 
If $\{F_1, \dots , F_r\}$ is an arbitrary family of $r$ linearly independent polynomials in $S_d$, then it follows from Lemmas \ref{ded}, \ref{case1}, \ref{case2}, \ref{Case3-1} and \ref{Case3-2}
that the number of common zeros in $\PP^m(\Fq)$ of $F_1, \dots , F_r$ is bounded above by 
$T_r(d,m)$ and, moreover, the bound is strict if $F_1, \dots , F_r$ do not have a common linear factor in $S$. 
Also by Lemma \ref{lub}, 
the bound $T_r(d,m)$ is attained. Thus  the theorem is proved. 
\end{proof}

Following Boguslavsky \cite{Bog}, we call a family $\{F_1, \dots , F_r\}$ of $r$ linearly independent polynomials in $S_d$ for which $|\V(F_1, \dots , F_r)| = T_r(d,m)$ 
to be a \emph{maximal $(r,m,d)$-configuration} over $\Fq$. Now Theorem \ref{mainthm} shows that for $d>1$ and $1\le r \le m+1$, a maximal $(r,m,d)$-configuration over $\Fq$ has a linear component in common so that $\V(F_1, \dots , F_r)$ contains a hyperplane. The example given by Lemma \ref{lub} has in fact a stronger property, namely, each of the $r$ polynomials is a product of $d$ distinct linear factors, and $(d-1)$ of these linear factors are common to all. It appears plausible that every maximal $(r,m,d)$-configuration satisfies such a property. We provide some evidence for this using rather sophisticated tools. For ease of reference, we first state a useful consequence proved in \cite[Cor. 6.6]{GL} of the Grothendieck-Lefschetz
Trace Formula, coupled with Deligne's Main Theorem concerning the so called
Riemann hypothesis for varieties over finite fields. We exclude the case $r=1$ since this it is covered by the result of Serre, viz., Theorem \ref{Serre}. 

\begin{proposition} 
\label{CorTrace}
Let $X$ be a projective algebraic variety defined over $\Fq$, and let 
$\bar{X} = X \otimes {\bar{\FF}}_q$ denote the corresponding variety over the algebraic closure of $\Fq$.
If $\dim \bar{X} = \delta$,  
then the limit 
$$
 \lim_{j\to \infty} 
\frac{ \left| {X(\FF_{q^j})}\right| } { q^{j\delta} } .
$$
exists and is equal to the number of irreducible
components of  $\bar{X}$ of dimension $\delta$. 
\end{proposition}

\begin{corollary}
\label{maxconfig}
Assume that  $1< d < q -1$ and $1 <  r \le m+1$. Let  $\{F_1, \dots , F_r\}$ be a maximal $(r,m,d)$-configuration over $\Fq$ as well as over every finite extension $\FF_{q^j}$ of $\Fq$. 
Then the projective variety $\V(F_1, \dots , F_r)$ is of codimension $1$ in $\PP^m$ and moreover, the corresponding projective variety over the algebraic closure of $\Fq$ has exactly $d-1$ irreducible components of codimension $1$ in $\PP^m$. 
\end{corollary}

\begin{proof}
Let $X=\V(F_1, \dots , F_r)$. 
By Theorem \ref{mainthm}, the polynomials $F_1, \dots , F_r$ have a common linear factor, and so $X$ contains a hyperplane. Also $|X(\Fq)| < p_m$. It follows that $\dim X = m-1$. Moreover, the limit as 
$j\to \infty$ of 
$ { \left| {X(\FF_{q^j})}\right| }/ { q^{j(m-1)} }$ is 
$$
 \lim_{j\to \infty} \frac{ (d-1) q^{ j (m-1) }  + q^{ j(m-2) }   + q^{j(m-3)} + \cdots + q^j + 1 + \lfloor q^{j(m-r)} \rfloor } { q^{j(m-1)} }
$$
and this is clearly equal to $d-1$. 
Thus Proposition~\ref{CorTrace} implies the desired result.
\end{proof}

To end this section, we remark that although Theorem \ref{mainthm} answers the question posed at the beginning of this paper when $d < q-1$ and $r\le m+1$, it does remain open in the remaining cases. It appears plausible that the same answer is true, more generally, when $d < q$ and $r\le m+1$, but some of the steps in our proof fail when $d=q-1$. 
It would be interesting to complete the result in the cases $d=q-1$ and $d=q$ 
as well, and with this hope, we have stated and proved some of the lemmas with a weaker assumption on $d$ (such as $d\le q$) whenever possible. Of course the more interesting case is that of $m+1 < r \le {{m+d}\choose{m}}$. As is shown in \cite{DG}, the TBC 
may not help here and a new guess may be needed. We venture to make the following guess for most (but not all) values of $r$ and $d$. 

\begin{conjecture}
\label{newconj}
Assume that $1<d <q$ and $1\le r \le { {m+d-1}\choose{m} }$. Then the maximum number of common zeros in $\PP^m(\Fq)$ that  a system of $r$ linearly independent homogeneous polynomials in $S_d$
can have is given by $H_r(d-1,m) + p_{m-1}$, where $H_r(d-1,m)$ is as in \eqref{Hrdm} except with $d$ replaced by $d-1$. Moreover, if the maximum number is attained 
by a system of $r$ linearly independent polynomials in $S_d$, then these polynomials have a common linear factor in $S$.  
\end{conjecture}

It may be worthwhile to note that the validity of the above conjecture implies Theorem \ref{mainthm} with, in fact, a slightly weaker hypothesis on $d$ (namely, $d < q$ rather than $d< q-1$); indeed, if $r\le m+1$, then 
$$
H_r(d-1,m) + p_{m-1} = (d-2)q^{m-1} + \lfloor q^{m-r}\rfloor + p_{m-1}  = (d-1)q^{m-1} + p_{m-2} + \lfloor q^{m-r}\rfloor .
$$
Moreover, Conjecture \ref{newconj} also implies Theorem \ref{HP} of Heijnen and Pellikaan \cite{HP} when $d < q-1$. To see this, 
suppose $f_1, \dots , f_r$ are linearly independent polynomials in $\Fq[x_1, \dots , x_m]$ of degree $\le d$, where $d < q-1$. Homogenize $f_1, \dots , f_r$ using the extra variable $x_0$  to obtain $r$ linearly independent polynomials, say $F_1, \dots , F_r$, in $S_d$. Let $\widetilde{F}_i := x_0F_i$ for $i=1, \dots , r$. Using Conjecture \ref{newconj}  applied to be $\widetilde{F}_1, \dots, \widetilde{F}_r$ in $S_{d+1}$, we see that $| \V(\widetilde{F}_1, \dots, \widetilde{F}_r)| \le H_r(d,m) + p_{m-1}$. On the other hand, intersecting $\V(\widetilde{F}_1, \dots, \widetilde{F}_r)$ with the hyperplane $\V(x_0)$ and its complement, we find that 
$
| \V(\widetilde{F}_1, \dots, \widetilde{F}_r)| = p_{m-1} + \left| Z(f_1, \dots , f_r)\right|$, 
where $Z(f_1, \dots , f_r)$ denotes the set of common zeros of $f_1, \dots , f_r$ in $\AA^m(\Fq)$. 
It follows that $|Z(f_1, \dots , f_r)| \le H_r(d,m)$. In a similar manner, the last assertion in Conjecture \ref{newconj} implies, using a linear change of coordinates, 
that the upper bound $H_r(d,m)$ is attained. 

In fact, a similar argument as in the above paragraph can be used to derive Theorem \ref{HP} in the case $r\le m+1$ from Lemmas \ref{ded} and \ref{lub}. But this is not so interesting since our proof of 
Lemma \ref{ded} uses Theorem \ref{HP}. It would, however, be interesting if a proof that 
Conjecture \ref{newconj}  holds in the affirmative can be obtained without using Theorem \ref{HP}. This is currently known in the case $d=2$ as a consequence (see \cite[Cor. 3.2]{DG}) of a result  of Zanella \cite[Thm. 3.4]{Z} for linear sections of quadratic Veronese varieties over finite fields. 

\begin{remark}
As outlined in \cite[\S 4.1]{DG}, results such as Theorem \ref{mainthm} can be used to explicitly determine several of the generalized Hamming weights of projective Reed-Muller codes $\PRM_q(d,m)$. Using further inputs from coding theory and a result of S{\o}rensen \cite{So}, one can also deduce information about some of the terminal higher weights of $\PRM_q(d,m)$. These can, in turn, be used to answer the question posed at the beginning of this paper for ``large'' values of $r$. We refer to \cite[\S 4]{DG2} for more on this. 
It appears noteworthy that by taking $r= \binom{m+d}{d}-2$, one can deduce that   
if $1<d< q-1$, then 
the Veronese variety $\Vmd$ does not contain a line. 
\end{remark}

\section*{Acknowledgments}
We would like to thank Sartaj ul Hasan for some initial discussions with the second named author on the question considered in this article. These discussions led to a somewhat simpler proof of Boguslavsky's theorem. We are also grateful to Peter Beelen for his interest in this work and some helpful comments.

\end{document}